\documentclass[a4paper]{amsart}
\usepackage{times}
\usepackage[T1]{fontenc}
\usepackage{amssymb, amsthm, amsmath}
\usepackage[active]{srcltx}
\usepackage{hyperref}
\usepackage{setspace}

\newtheorem{theorem}{Theorem}[section]

\newtheorem{conjecture}[theorem]{Conjecture}
\newtheorem{corollary}[theorem]{Corollary}

\newtheorem{fact}[theorem]{Fact}
\newtheorem{lemma}[theorem]{Lemma}

\newtheorem{proposition}[theorem]{Proposition}

\theoremstyle{definition}

\newtheorem{remark}[theorem]{Remark}

\newcommand{\Qq}{{\mathbb{Q}}}

\newcommand{\CL}{{\mathcal L}}

\newcommand{\CN}{{\mathcal N}}

\newcommand{\CM}{{\mathcal M}}

\newcommand{\CO}{{\mathcal O}}

\renewcommand{\phi}{\varphi}

\long\def\symbolfootnote[#1]#2{\begingroup%
\def\thefootnote{\fnsymbol{footnote}}\footnote[#1]{#2}\endgroup}

\makeatletter

\def\Ind#1#2{#1\setbox0=\hbox{$#1x$}\kern\wd0\hbox to 0pt{\hss$#1\mid$\hss}
\lower.9\ht0\hbox to 0pt{\hss$#1\smile$\hss}\kern\wd0}

\def\Notind#1#2{#1\setbox0=\hbox{$#1x$}\kern\wd0\hbox to 0pt{\mathchardef
\nn=12854\hss$#1\nn$\kern1.4\wd0\hss}\hbox to
0pt{\hss$#1\mid$\hss}\lower.9\ht0 \hbox to
0pt{\hss$#1\smile$\hss}\kern\wd0}

\def\dprk{\textrm{dp-rk}}

\makeatother 

\title{A conjectural classification of strongly dependent fields}

\date{\today}
\author[Y. Halevi]{Yatir Halevi$^*$}
\thanks{$^*$Partially supported by the European Research Council grant 338821, by ISF grant No. 181/16 and ISF grant No. 1382/15.}
\address{$^*$Einstein Institute of Mathematics\\
	The Hebrew University of Jerusalem\\
	Givat Ram\\
	Jerusalem 91904\\
	Israel\\}
\email{yatir.halevi@mail.huji.ac.il}
\urladdr{http://ma.huji.ac.il/\textasciitilde yatirh/}

\author[A. Hasson]{Assaf Hasson$^\dagger$}
\thanks{$^\dagger$ Supported by ISF grant No. 181/16}
\address{$^\dagger$Department of mathematics\\
	Ben Gurion University of the Negev\\
	Be'er Sehva\\
	Israel} \email{hassonas@math.bgu.ac.il} \urladdr{http://www.math.bgu.ac.il/\textasciitilde hasson/}

\author[F. Jahnke]{Franziska Jahnke$^\ddagger$}
\thanks{$^\ddagger$ Partially supported by SFB 878}
\address{$^\ddagger$ Westf\"{a}lische Wilhelms-Universit\"{a}t M\"unster, 
	Institut f\"ur Mathematische Logik und Grundlagenforschung, 
	Einsteinstr. 62, 48149 M\"unster,
	Germany}\email{franziska.jahnke@uni-muenster.de}

\begin{document}

\begin{abstract}
	We survey the history of Shelah's conjecture on strongly dependent fields, give an equivalent formulation in terms of a classification of strongly dependent fields and prove that the conjecture implies that every strongly dependent field has finite dp-rank.
\end{abstract}

\maketitle

\section{introduction}

The algebraic classification of algebraic structures satisfying various model theoretic assumptions is a fruitful line of research in model theory for almost half a century. It originates with Macintyre's proof, \cite{Mac71}, that $\omega$-stable fields are algebraically closed, later generalized by Cherlin and Shelah to superstable fields, \cite{ChSh}, and with Reineke's proof, \cite{ReiMG}, that $\omega$-stable minimal groups are commutative. The classification of o-minimal ordered rings (they are real closed fields, \cite{PilStei1}), of weakly o-minimal fields (they are also real closed, \cite{MacMaSt}),  of C-minimal fields (they are algebraically closed valued fields, \cite{HasMac}) and of $\omega$-stable groups of ranks 2 (they are solvable by finite, \cite{Ch1}) are just a few examples of other results of a similar flavour. 

Some of the best known and long standing open problems in model theory concern such classification problems: 

\begin{conjecture}
	\begin{enumerate}
		\item \textbf{The algebraicity conjecture} Every simple group of finite Morley rank is an algebraic group. 
		\item \textbf{The stable fields conjecture}: Every infinite  stable field is separably closed. 
		\item \textbf{The simple fields conjecture} Every infinite (super)-simple field is a (perfect) PAC field. 
	\end{enumerate}
\end{conjecture}

For several decades little (if any) progress has been made on the last two of these conjectures (Duret, in \cite{duret}, has shown that the simple fields conjecture implies the stable fields conjecture because a PAC field is stable if and only if it is separably closed), and the algebraicity conjecture is still wide open. 

In the early 2000s Shelah started exploring the problem of finding, among theories without the independence property (also known as dependent theories or NIP theories), a dividing line analogous to super-stability among stable theories (\cite{Sh783}, \cite{Sh863}). From this research Shelah extracted several sub-classes of dependent theories, of which strongly dependent theories seemed one of the most natural and interesting division lines. 

As a test for the appeal of this division line Shelah suggested in \cite[Conjecture 5.34]{Sh863} the conjecture that infinite strongly dependent fields are algebraically closed, real closed, or elementary equivalent to a field admitting a valuation $v$, with strongly dependent value group and residue field, that eliminates field quantifiers in some Denef-Pas language, or a finite algebraic extension of the latter.

It is widely believed that every valued field admitting elimination of field quantifiers in a Denef-Pas language is henselian\footnote{S. Rideau and I. Halupczok informed us (private communication) that this claim is proved in equicharacteristic $0$ in a yet unpublished work with R. Cluckers.}. Thus, with time, 
this conjecture converged into the following more algebraic formulation: 
\begin{conjecture}\label{Shelah}
	Any infinite strongly dependent field which is neither real closed nor algebraically closed admits a non-trivial henselian valuation. 
\end{conjecture}


As the understanding of dependent fields expanded, it was noticed that there are no known counter examples to the much stronger: 
\begin{conjecture}\label{ShelahNIP}
	Any infinite dependent field which is neither real closed nor separably closed admits a non-trivial henselian valuation. 
\end{conjecture}


In the present short note we will explain the statement of these conjectures, and why Conjecture \ref{ShelahNIP} may provide new insight into the stable fields conjecture. Finally, we will show (Theorem \ref{T:main}) that Conjecture \ref{Shelah} is equivalent an algebraic classification of strongly dependent fields and implies that every strongly dependent field is of finite dp-rank.

\section{Background}

Throughout this note (strong) dependence will be used as a black box, and we refer interested readers to \cite{SiBook} for details. Important for our needs will be the dp-rank, denoted by $\textrm{dp-rk}$,  associated with any dependent theory: it is the supremum on all cardinals $\kappa$ such that for any sequence of mutually indiscernible sequences $\langle I_j:j<\kappa\rangle$ of (possibly infinite) tuples (of a saturated model of $T$) there exists a singleton $c$ such that $I_j$ is not indiscernible over $c$ for all $j<\kappa$. In this terminology, $T$ is strongly dependent if $\textrm{dp-rk}(T)\le  \aleph_0$ but one can not find $\aleph_0$ such mutually indiscernible sequences. It is \emph{dp-minimal} if $\textrm{dp-rk}(T)=1$. It follows from \cite[Exercise 2.14]{SiBook} (\cite[Claim 3.14]{Sh783}) that a theory $T$ is (strongly) dependent if and only if the canonical expansion of $T$ by all imaginary sorts, $T^{eq}$, is (strongly) dependent.

Examples of dp-minimal theories include all theories of linear orders, strongly minimal theories, (weakly) o-minimal theories, $p$-adically closed fields (and their finite extensions), algebraically closed (or real closed) valued fields and more. The canonical example of a dependent theory of fields that is not strongly dependent is that of separably closed fields that are not algebraically closed. It is well known, \cite{GurSch}, that all ordered abelian groups are dependent. Among those, the dp-minimal are those groups $G$ such that $[G:pG]<\infty$ for all primes $p$, \cite[Proposition 5.27]{JaSiWa2015}, and the strongly dependent ones are those satisfying that $|\{p: [G:pG]=\infty, p \text{ prime}\}|<\infty$ and $G$ has finite spines \cite[Theorem 4.20]{HaHa} or \cite[Section 6]{farre}. 

A crucial fact in the study of Shelah's conjecture is : 
\begin{fact} \cite[Corollary 3.24]{SiBook} \cite[Observation 3.8]{alfalex}
	Let $\CM$ be a dependent (resp. dp-minimal or strongly dependent) structure. Then $\CM^{sh}$ is dependent (resp. dp-minimal or strongly dependent), where $\CM^{sh}$ is the Shelah expansion of $\CM$, namely the expansion of $\CM$ by all externally definable sets.
\end{fact}

In the above, recall that , if $\CM$ is any structure $S\subseteq M^k$ (some $k$) is \emph{externally definable} if there exists $\CM\prec \CN$ and an $N$-definable (with parameters) $\hat S\subseteq N^k$ such that $S=\hat S\cap M^k$. Note that, for example, if $\CM$ is a linearly ordered structure, then, by compactness, any cut in $M$ is externally definable.  It is well known (and easy to check) \cite[Corollary 8.3]{TZ} that all externally definable sets in a stable structure are, in fact, internally definable. So, for example, no stable field admits an externally definable (non-trivial) valuation. 

The above fact is particularly useful when studying valued fields. We remind that if $K$ is a field, $\Gamma$ an ordered abelian group, then a valuation $v:K^\times\to \Gamma$ is a group homomorphism satisfying $v(x+y)\ge \min\{v(x),v(y)\}$ where $v(0):=\infty$. The set $\CO_v:=\{x:v(x)\ge 0\}$ is a ring (called a the valuation ring) and $\mathfrak M_v:=\{x: v(x)>0\}$ is its unique maximal ideal. The field $Kv:=\CO_v/\mathfrak M_v$ is called the residue field. The value group $\Gamma$ is also denoted by $vK$. Note that $vK\cong K^\times/\CO^\times_v$ (as ordered groups\footnote{The order on $K^\times/\CO_v^\times$ is given by $[x]\ge [y]$ if $xy^{-1}\in \CO_v$.}) implying that giving a valuation $v$ on $K$ is the same as specifying its valuation ring.

A valuation $w:K^\times\to \Gamma'$ is a coarsening of $v$ if $\CO_v\subseteq \CO_w$. There is a natural bijection between valuations coarsening $v$ and convex subgroups of the valuation group, $vK$ (see \cite[Lemma 2.3.1]{EnPr}). So, e.g., given a field $K$ (in some language $\CL$), if a valuation $v$ is definable on $K$ then any valuation on $K$ coarsening $v$ is externally definable. It follows that if $(K,v)$ is a dependent valued field then so is $(K,w)$ whenever $w$ is a coarsening of $v$. 

Finally, we recall that by a theorem of Chevalley (see \cite[Theorem 3.1.1]{EnPr}), if $(K,v)$ is a valued field and $K\subseteq L$ is any field extension, then $v$ can be extended to a valuation on $L$. We remind that $v$ is called \emph{henselian} if it has a unique extension to any algebraic extension. Classical results (starting with \cite{AxKo}) show that under suitable assumptions the theory of a henselian valued field $(K,v)$ is completely determined by the theories of the value group, $vK$, and the residue field, $Kv$. In particular, this is true for strongly dependent valued fields (in a suitable language), see \cite[Corollary 4.5]{HaHaQE}. 

Among henselian valuations on a non-separably closed field, $K$, there exists a canonical one, denoted $v_K$. It is the coarsest henselian valuation $v$ on $K$ with separably closed residue field -- if such a $v$ exists -- and the finest henselian valuation on $K$, otherwise (see, \cite[Section 4.4]{EnPr} for details). It is non-trivial if and only if $K$ is a henselian field (i.e., admits some henselian valuation). It follows that a stable field that is not separably closed is not henselian. Indeed, if $K$ is not separably closed henselian and $Kv_K$ is separably closed, then by \cite[Theorem 3.10]{JahKoeDef} $K$ admits a definable non-trivial valuation, implying that $K$ is unstable (since stable structures can not interpret an infinite linear order). If $Kv_K$ is not separably closed then $v_K$ is externally definable (\cite[Theorem A]{JanNIP}), which is again impossible if $K$ is stable. Moreover,  the above argument shows that a non-separably closed stable field cannot be $t$-henselian, that is, elementarily equivalent to a henselian field, and since any finite extension of a stable field is itself stable, the above shows that no finite extension of a stable field is $t$-henselian. 

In particular, Shelah's Conjecture \ref{ShelahNIP} implies the stable fields conjecture. Observe that a weaker form of Conjecture \ref{ShelahNIP} suffices, namely: 
\begin{conjecture}\label{ShelahWeak}
	Any infinite dependent field $K$ is either real closed, separably closed, or admits a definable non-trivial valuation. 
\end{conjecture}

We note that, by using a result from a subsequent paper (see \cite[Proposition 5.3]{NIPthen}) we get: 
\begin{fact}
	The following statements are equivalent:
	\begin{enumerate}
		\item Every infinite dependent field is either real closed, separably closed or admits a henselian valuation. 
		\item Every infinite dependent field is either real closed, separably closed or admits a \emph{definable} henselian valuation.
	\end{enumerate}
\end{fact}
So Conjecture \ref{ShelahWeak} follows from Conjecture \ref{ShelahNIP}. \\

In \cite[Proposition 5.4]{NIPthen} we show that, assuming Conjecture \ref{ShelahNIP}, if $(K,v)$ is a dependent valued field then $v$ is henselian. In particular, Conjecture \ref{ShelahNIP} implies that if $K$ is a dependent field and $v$ is any (externally) definable valuation on $K$ then $v$ is henselian. We do not know whether Conjecture \ref{ShelahNIP} is equivalent to Conjecture \ref{ShelahWeak}. 

It follows from the above discussion that a plausible strategy for proving Conjecture \ref{ShelahNIP} would be to prove separately Conjecture \ref{ShelahWeak} and the conjecture that if $(K,v)$ is dependent then $v$ is henselian (or even the weaker statement that any definable valuation on a dependent field is henselian). 

There are reasons to believe that of the two statements above, Conjecture \ref{ShelahWeak} is the more challenging (if only because it implies the notorious stable fields conjecture). There are, however, several techniques for constructing definable valuations. Let us describe briefly one such approach not assuming 
$t$-henselianity of the field (or any of its finite extensions) which was
also explored in Dupont's PhD thesis (see \cite{DuHaKu}). A valued field is $p$-henselian if its valuation extends uniquely to every finite extension of degree $p$. A field is $p$-henselian if it supports a non-trivial $p$-henselian valuation.

In \cite{Koe95}, Koenigsmann characterizes $p$-henselian fields (of
characteristic different from $p$ and containing a primitive $p$th root of unity) as those fields with
$G:=(K^\times)^p\subsetneq K^\times$ in which the collection $\CN_G$ of sets of the form $a(G+1)\cap b(G+1)$ is a basis for a $V$-topology, namely, a field topology equivalent to the one induced by the open balls of some non-trivial valuation (see \cite[Appendix B]{EnPr} for details).  
There is also a corresponding version for characteristic $p$.
As already mentioned, by \cite{JahKoe}, every $p$-henselian field admits a non-trivial 
$0$-definable henselian valuation.
Since, by standard Galois theoretic arguments, every henselian field (which is neither separably closed nor real closed) has a finite extension satisfying the above assumptions, Conjecture \ref{ShelahNIP} implies that every infinite dependent field that is neither separably closed nor real closed admits a definable valuation.

Note that given a definable subgroup $G$ of $K$, as above, the statement that the collection $\CN_G$ is a basis for a $V$-topology can be expressed by a first order sentence, $\Psi$, (without parameters). 
Thus, proving that an infinite field not satisfying $\Psi$ has the independence property (i.e., is not dependent) seems like a possible approach for proving Conjecture \ref{ShelahWeak} and an important first step in proving Conjecture \ref{ShelahNIP}. 




\section{Strongly dependent fields}
There are good reasons to believe that Conjecture \ref{Shelah} (at least restricted to the finite dp-rank case) may be more accessible than the full conjecture for dependent fields. For example, the stable fields conjecture for fields of finite dp-rank has a surprisingly short proof by showing that such fields are super-stable (of finite $\mathrm{U}$-rank), see \cite[Proposition 7.2]{HalPal}. More significant is Johnson's proof in \cite{johnsonpaper} of Shelah's conjecture for dp-minimal fields. Johnson's proof proceeds, as suggested above: constructing a $V$-topology, and then showing that this topology must come form a definable henselian valuation (that a dp-minimal valued field is henselian was proved independently, using different methods, in \cite{JaSiWa2015}). In his proof Johnson introduces a new approach for constructing a field topology -- exploiting heavily the dp-minimality of the field. In his Phd thesis, Sinclair attempts to generalise Johnson's construction to fields of finite dp-rank but leaves major open problems, see \cite{Peter}.

While Conjecture \ref{Shelah} remains wide open, despite Johnson's breakthrough, we can give a more precise statement of the conjecture. In \cite{JohnDPMin} Johnson proved:
\begin{theorem}[Johnson]\label{classification}\cite[Theorem 9.7.2]{JohnDPMin}
	A field $K$ is dp-minimal if and only if $K$ is perfect and there exists a valuation $v$ on $K$ such that: 
	\begin{enumerate}
		\item $v$ is henselian. 
		\item $v$ is defectless (i.e., any finite valued field extension $(L,v)$ of $(K,v)$ is defectless, i.e., satisfies $[L:K]=[vL:vK][Lv:Kv]$). 
		\item The residue field $Kv$ is either an algebraically closed field of characteristic $p$ or elementarily equivalent to a local field of characteristic $0$. 
		\item The value group $vK$ is almost divisible, i.e., $[vK: n(vK)]<\infty$ for all $n$. 
		\item If $\mathrm{char}(Kv)=p\neq 0$  then $[-v(p),v(p)]\subseteq p(vK)$. 
	\end{enumerate}
\end{theorem}

\begin{remark}
Johnson's formulation of \cite[Theorem 9.7.2]{JohnDPMin} requires  $K$ to be sufficiently saturated and as a result in (3) above he gets the residue field to be algebraically closed of characteristic $p$ or a local field of characteristic 0. A proof similar to the one below shows that our reformulation of Johnson's result can be obtained from the original theorem. 
\end{remark}

We suggest that with the obvious adaptations, Johnson's theorem characterises all strongly dependent fields: 
\begin{conjecture}\label{NewConj}
	A field $K$ is strongly dependent if and only if it is perfect and there exists a valuation $v$ on $K$ such that 
\begin{enumerate}
		\item $v$ is henselian. 
		\item $v$ is defectless.
		\item The residue field $Kv$ is either an algebraically closed field of characteristic $p$ or elementarily equivalent to a local field of characteristic $0$. 
		\item The valuation group $vK$ is strongly dependent. 
		\item If $\mathrm{char}(Kv)=p\neq 0$  then $[-v(p),v(p)]\subseteq p(vK)$.
\end{enumerate}
\end{conjecture}

In the present note we show
\begin{theorem}\label{T:main}
	Conjecture \ref{Shelah} is equivalent to Conjecture \ref{NewConj}. 
\end{theorem}

It is clear that Conjecture \ref{NewConj} implies Conjecture \ref{Shelah}. Indeed, let $v$ be the valuation provided by Conjecture \ref{NewConj}. If $v$ is trivial then, by (3), $K$ is algebraically closed, real closed or elementarily equivalent to a finite extension of $\Qq_p$. If $v$ is non-trivial Shelah's conjecture follows from (1). So we focus on proving the other direction. 

An \emph{angular component map} on a valued field $(K,v)$ is a multiplicative group homomorphism $ac:K^\times \to Kv^\times$ such that $ac(a)=res(a)$ whenever $v(a)=0$, we extend it to $ac:K\to Kv$ by setting $ac(0)=0$. Every valued field $(K,v)$ has an elementary extension with an angular component map on it, see \cite[Corollary 5.18]{vd}. An \emph{ac-valued field} is a valued field equipped with an angular component map. The $3$-sorted language of valued fields augmented  by a function symbol \emph{ac} for the angular component map is called the Denef-Pas language. We will repeatedly use the following:
\begin{fact}\label{F:transfer-strongly}
Let $(K,v)$ be an ac-valued field eliminating field quantifiers. If $Kv$ and $vK$ are strongly dependent then so is $(K,v,ac)$.
\end{fact}
\begin{proof}
This is basically \cite[Claim 1.17(2)]{Sh863} combined with 
\cite[Theorem 5]{DelHenselian}.
\end{proof}

That any field satisfying conditions (1)-(5) of Conjecture \ref{NewConj} is strongly dependent is, esssentially, the content of the main results of \cite{HaHaQE}. Indeed, if $v$ is trivial, then by (3), $K$ is strongly dependent (in fact, dp-minimal). Otherwise, pass to an $\aleph_1$-saturated elementary extension of $(K,v)$, so we may assume that $(K,v)$ is ac-valued. Conditions (1)--(5) are elementary and thus still hold. Any field satisfying conditions (1) and (2) is algebraically maximal. Condition (3) implies that the residue field is strongly dependent, and therefore -- if it is of characteristic $p>0$ -- has no finite extensions of degree $p$ by \cite[Corollary 4.4]{KaScWa}. Condition (5) implies, with all of what we have already said, that $(K,v,ac)$ eliminates field quantifiers by \cite[Theorem 1]{HaHaQE} (in the notation there, $(K,v)\models T_1$). Consequently, by Fact \ref{F:transfer-strongly}, $(K,v,ac)$, and thus $(K,v)$, is strongly dependent. Hence $K$ is strongly dependent in the language of rings.

The following observation, with the obvious adaptations is true for any dependent field (assuming Conjecture \ref{ShelahNIP}), see \cite[Proposition 5.3]{NIPthen}. For completeness we give the proof for strongly dependent fields.

\begin{lemma}\label{L:Jahnke}
	Assume Conjecture \ref{Shelah}.  Then any infinite strongly dependent $K$ admits a henselian valuation $v$ (possibly trivial) such that $Kv$ is real closed, algebraically closed or finite. In particular, any infinite 
strongly dependent field that is neither real closed nor algebraically closed 
admits a ($\emptyset$-)definable henselian valuation 
(in the language of rings).
\end{lemma}
\begin{proof}
Let $K$ be an infinite strongly dependent field. We first show the existence
of a henselian valuation $v$ on $K$ with $Kv$ finite, real closed or algebraically closed.
We may assume that $K$ is neither algebraically closed nor real closed (otherwise take $v$ to be the trivial valuation). If $K$ admits a henselian valuation, $v$, with $Kv$  separably closed, then by \cite[Proposition 5.2]{HaHa} 
$Kv$ is strongly dependent and thus by \cite[Proof of Claim 5.40]{Sh863} already
algebraically closed. So we may assume that this is not the case, i.e. that every non-trivial henselian valuation on $K$ has non separably closed residue field.
	By \cite[Theorem 4.4.2]{EnPr}, there exists a finest henselian valuation with non separably closed residue field, the canonical valuation $v_K$.  Notice that $Kv_K$ is non henselian. Indeed, if it were henselian then composing the corresponding place with $K\to Kv_K$ would yield a henselian valuation on $K$ strictly finer than $v_K$, contradiction.
	 Since $(K,v_K)$ is strongly dependent (\cite[Theorem 2]{HaHa}) so is $Kv_{K}$, and applying Conjecture \ref{Shelah} again, we get that it is finite or real closed. 

The second statement follows since any field $K$ that admits a
non-trivial henselian valuation with residue field $Kv$ finite, real closed or
algebraically closed supports a nontrivial $\emptyset$-definable 
henselian valuation 
(see \cite[Proposition 3.1, Theorem 3.10 and Corollary 3.11 respectively]{JahKoeDef}).
\end{proof}

Recall that a \emph{$\wp$-adically closed field} is a $p$-valued field of $p$-rank d (in the sense of \cite{PR84}) which does not admit any proper algebraic extension of the same $p$-rank. This is a first order property by \cite[Theorem 3.1]{PR84}. The following is folklore.
	\begin{fact}\label{F:p-adically-elem-fiQp}
	Every $\wp$-adically closed field is elementary equivalent to a finite extension of a $p$-adic field,  $\mathbb{Q}_p$.
	\end{fact}
	\begin{proof}
	Let $(K,v)$ be a $\wp$-adically closed field, which we may assume to be sufficiently saturated and thus containing an isomorphic copy of $\mathbb{Q}_p$. Let $L$ be the intersection of $K$ and $\mathbb{Q}_p^{alg}$ (taken inside $K^{alg}$). It is also $\wp$-adically closed of the same $p$-rank as $K$ by \cite[Theorem 3.4]{PR84}. Thus $L$ is elementary equivalent to $K$ by model completeness \cite[Theorem 5.1]{PR84}. On the other hand, $L$ is a finite extension of $\mathbb{Q}_p$  since it has finite ramification index and finite inertia degree, see \cite[page 15]{PR84}.
	\end{proof}

We can now prove the left to right direction of our main result: 

\begin{proposition}\label{P:shelah-implies-new}
	Assume Conjecture \ref{Shelah}, then Conjecture \ref{NewConj} holds. 
\end{proposition}
\begin{proof}
	Let $K$ be an infinite strongly dependent field. If $K$ is real closed or algebraically closed, take $v$ to be the trivial valuation, and there is nothing to prove. So we assume this not to be the case. By Lemma \ref{L:Jahnke} there exists a henselian valuation $v$ on $K$ such that $Kv$ is real closed, algebraically closed or finite. 
	
	By \cite[Theorem 4.3.2]{JohnDPMin} $(K,v)$ is defectless and by Theorem 5.14 of the same paper $(K,v)$ is strongly dependent. Thus, $vK$ is strongly dependent. If $Kv$ is infinite then Conjecture \ref{NewConj}(5) follows from \cite[Lemma 2.7]{HaHaQE}. 
	
	Assume $Kv$ is finite. Thus $[0,v(p)]$ is finite by \cite[Proposition 5.13]{HaHa} and so $(K,v)$ is a $p$-valued field. Let $\Delta$ is the convex subgroup of $vK$ generated by $v(p)$. Then the residue field, $K_1$, associated with the coarsening $w:K\to vK/\Delta$ is $\wp$-adically closed by 
\cite[Section 2.2, Theorem 3.1]{PR84} and elementary equivalent to a finite extension of $\mathbb{Q}_p$ by Fact \ref{F:p-adically-elem-fiQp} (i.e. a non-archimedean local field of characteristic $0$).
	
\end{proof}

Finally, we point out that Conjecture \ref{Shelah} has further implications on strongly dependent fields: 

\begin{proposition}
	Assume Conjecture \ref{Shelah}. Then every strongly dependent field has finite dp-rank. 
\end{proposition}
\begin{proof}
	Since the conclusion of the lemma is well known for algebraically closed, real closed and finite fields, we may assume that this is not the case. Fix a strongly dependent field $K$. Since $\dprk(K)$ is invariant under elementary equivalence, there is no harm assuming that $K$ is $\aleph_1$-saturated. Fix  a henselian valuation $v$ on $K$, as provided by Conjecture \ref{NewConj}. If $v$ is trivial, there is nothing to show. 
	
	So $K$ is an $ac$-valued field. By \cite[Theorem 5.14]{HaHa} we know that $(K,v)$ is strongly dependent, and as in the paragraph after Fact \ref{F:transfer-strongly}, $(K,v,ac)$ is strongly dependent. By definition, the dp-rank of $(K,v,ac)$ is at least that of $(K,v)$, so it will suffice to prove that the dp-rank of $(K,v,ac)$ is finite. 
	
	In case $(K,v)$ is of equi-characteristic $0$, the lemma follows from \cite[Theorem 7.6(2)]{Chernikov}, the fact that $vK$ has finite dp-rank \cite[Theorem 1]{HaHa} and the dp-minimality of $Kv$. 
	
	The general case follows from, essentially, the same argument. Here are the details: Since $(K,v)$ is strongly dependent and $Kv$ is infinite, $(K,v,ac)$ eliminates field quantifiers (see \cite[Theorem 1]{HaHaQE}, it is a model of $T_1$ in the notation there) and it follows that $Kv$ and $vK$ are stably embedded. So their respective dp-ranks (as pure structures) are the same as their dp-ranks with the structure induced from $(K,v,ac)$. In particular, $\dprk(vK)<\infty$ by \cite[Theorem 1]{HaHa} and $\dprk(Kv)\le 1$ by the choice of $v$. 
	
	It follows from \cite[Lemma 7.12]{Chernikov} that there is no inp-pattern of depth $\omega$ whose formulas are of the form $\bigvee\limits_{i<n} (\xi_i(x)\land \rho_i(x))$, where $\xi_i(x):=\bigwedge \xi_i^j(v(x-c_i^j), d^j_i))$ with $\xi_i^j$ in the language of ordered abelian groups, and $\rho_i:=\bigwedge \rho_i^j(ac(x-c_i^j), e_i^j)$ with $\rho_i^j$ in the language of rings. 
	
	In the equi-characteristic 0 case, the result now follows from cell decomposition, which implies that any formula is equivalent to one as above. In general, it follows from elimination of field quantifiers, and \cite[Theorem 5]{DelHenselian} any type over a model is isolated by formulas as above. By compactness, every formula is equivalent to one of the above form, so the result follows. 	
	
\end{proof}

\begin{remark}
	In Sinclair's PhD thesis he shows that, in fact, if $(K,v,ac)$ is a strongly dependent henselian valued field with $Kv$ infinite and with an angular component map, then $\dprk(K,v,ac)=\dprk(vK)+\dprk(Kv)$ where on the left hand side the dp-rank is computed with respect to the Denef-Pas language, and on the right hand side it is computed with respect to the language of ordered abelian groups and the language of rings respectively \cite[Theorem 2.2.7]{Peter}. 
\end{remark}

It follows from \cite[Corollary 4.4]{HaHaQE} that, combined with the main result of \cite{HaHa}, Conjecture \ref{NewConj} determines all possible first order theories of strongly dependent fields. 
The following is a special case of an unpublished result of S. Anscombe and the third author:\\

\begin{theorem}
	Let $K$ be an infinite strongly dependent field and let $v_K$ be the (possibly trivial) canonical henselian valuation on $K$. Assuming Shelah's conjecture \ref{Shelah}, one the following holds: 
\end{theorem}
\begin{enumerate}
	\item $Kv_K$ is real closed or algebraically closed of characteristic $0$ and $K\equiv \mathbb R((\Gamma))$ or $K\equiv \mathbb C((\Gamma))$ (as fields) where $\Gamma\equiv v_KK$ (as ordered abelian groups) and $\Gamma$ is strongly dependent. 
	\item $\mathrm{char }K=p>0$, $Kv_K$ is algebraically closed, $(K,v_K)$ is tame Kaplansy and $K\equiv \overline{\mathbb F}_p ((\Gamma))$ (as fields) where $\Gamma\equiv v_KK$ (again, as ordered abelian groups) and $\Gamma$ is strongly dependent. 
	\item $\mathrm{char}(K,Kv_K)$ is of mixed characteristic $(0,p)$ and $Kv_K$ is finite. 
	Then $K\equiv Q((\Gamma))$ (as fields) where $Q$ is a finite extension of $\Qq_p$ and $\Gamma$ is strongly dependent. 
	\item $\mathrm{char}(K,Kv_K)$ is of mixed characteristic $(0,p)$ and $Kv_K$ is infinite. In that case $Kv_K\models ACF_p$ and $K\equiv L((\Gamma))$ (as fields) where $L$ is a field admitting a rank 1 valuation $v$ turning it into a mixed characteristic tame Kaplansky field, with residue field as in (2) above and $\Gamma$ is strongly dependent
\end{enumerate}

Moreover, any of the fields described in clauses $(1)-(4)$ is strongly dependent.
\begin{proof}
	Note that by \cite[Theorem 5.14]{HaHa} $(K,v_K)$ is strongly dependent. 

In case $\mathrm{char}(K,Kv_K)=(0,0)$, 
it follows from (the proof of) Lemma \ref{L:Jahnke} that $Kv_K$ is real closed or algebraically closed. That $\Gamma:=vK_K$ 
is strongly dependent follows from \cite[Proposition 5.11]{HaHa}. Since $Kv_K$ is strongly dependent by assumption, the fact that $K$ is as in (1) follows from Ax-Kochen-Ershov (see also, \cite[Corollary 4.4]{HaHaQE}). 

Assume $K$ is strongly dependent, Shelah's conjecture holds and 
$\mathrm{char}(K,Kv_K)=(p,p)$
for a prime $p>0$. The proof that then $K$ is as described in (2) is similar, 
noting that by \cite[Lemma 5.12]{HaHaQE} $(K,v_K)$ is algebraically maximal Kaplansky (see \cite[Definition 5.10]{HaHa} for the details). 
	
In the case where $\mathrm{char}(K,Kv_K)=(0,p)$ is mixed, we separate into cases
depending on whether $Kv_K$ is finite or infinite.
In case $Kv_K$ is finite, consider the composition of places 
\[
K\xrightarrow{\Delta/\Delta_1}L\xrightarrow{\Delta_1/\Delta_0} K'\xrightarrow{\Delta_0}Kv_K.
\]
where $\Delta:=v_KK$ and $\Delta_0<\Delta_1\le \Delta$, where $\Delta_0$ is the maximal convex subgroup not containing $v_K(p)$ and $\Delta_1$ is the minimal convex subgroup containing $v_K(p)$. 
Since $Kv_K$ is finite, the proof of Proposition \ref{P:shelah-implies-new}
shows that $\Delta_0$ is trivial and $L$ is 
$\wp$-adically closed. That $K$ is as
in (3) now follows by Ax-Kochen-Ershov applied to $K\xrightarrow{\Delta/\Delta_1} L$. That $\Gamma:=\Delta/\Delta_1$ is strongly dependent follows, again, by \cite[Proposition 5.11]{HaHa}. 

For $Kv_K$ infinite, 
we consider the same decomposition as above. As before, we have 
$K\equiv L((\Delta/\Delta_1))$. By the equicharacteristic $p$ case above, 
$K'\equiv \overline{\mathbb F}_p((\Delta_0))$ and $L$ is a rank 1, mixed characteristic tame Kaplansky field (by \cite[Theorem 5.13]{HaHa}). Since algebraically maximal Kaplansky fields admit elimination of field quantifiers in the Denef-Pas language, the theory of $L$ is completely determined by the theory of $\Delta_1/\Delta_0$ (in the language of ordered abelian groups) and of $K'$ (in the language of rings). Moreover, by \cite[Corollary 2.15]{FKVBad} and \cite[Theorem 5.3]{KuPaRo} for every $p$-divisible ordered abelian group $\Gamma$, and any strongly dependent field $k$ of characteristic $p$,  there exists a (up to elementary equivalence) unique tame Kaplansky field $(L,v)$ of mixed characteristic $(0,p)$ with $vL\equiv\Gamma$ and $Lv\equiv k$.

The moreover follows by \cite[Lemma 4.8]{HaHaQE}.
\end{proof}

Recall that a field $K$ is called bounded if $K$ has only finitely many Galois extensions of degree $n$ for each $n$. The above theorem shows that --
assuming Shelah's conjecture -- boundedness is the dividing line between dp-minimal and non-dp-minimal strongly dependent fields:
\begin{corollary}
Assume Shelah's conjecture \ref{Shelah} and let $K$ be a strongly dependent field. Then $K$ is dp-minimal if and only if $K$ is bounded.
\end{corollary}
\begin{proof} 
Both directions are straightforward, although the proofs are lengthy.
We first use Johnson's classification to show that dp-minimal fields
are bounded. Let $K$ a dp-minimal field. Since finite fields are bounded,
we may assume that $K$ is infinite. Thus, 
by \cite[Theorem 9.7.2]{JohnDPMin} (more precisely, its rephrasing as 
Theorem \ref{classification} above), $K$ admits a defectless henselian 
valuation $v$ with residue field $Kv$ algebraically closed of 
characteristic $p$ or elementarily equivalent to a local field of
characteristic $0$. Moreover, we have $[vK : n(vK)] < \infty$ for all
$n \in \mathbb{N}$ and, in case 
$\mathrm{char}(Kv)=p \neq 0$, we have $[-v(p), v(p)] \subseteq
p(vK)$. Note that in each case, $Kv$ is bounded.
Let $K^\mathrm{alg}$ denote an algebraic closure of $K$, let $u$ denote the
(unique) extension of $v$ to $K^\mathrm{alg}$ and write $G_K$
for the absolute Galois group of $K$. We will argue using the
ramification subgroup $R(u/v)$ and the inertia subgroup $I(u/v)$ of $G_K$.
For an introduction to these groups and their properties, 
see  \cite[Ch. 5]{EnPr}.
We now split into cases depending on whether $(K,v)$ has mixed characteristic or not.
\begin{enumerate}
\item If $\mathrm{char}(K)=\mathrm{char}(Kv)$, 
we claim that $R(u/v)$ is trivial. This is obvious in case $\mathrm{char}(Kv)=0$ 
(see \cite[Theorem 5.3.3]{EnPr}), and follows from the fact that $K$ admits no Galois extensions of degree
divisible by
$p$ in case $\mathrm{char}(K)=p>0$ (\cite{KaScWa}) since $R(u/v)$ is a pro-$p$ group 
(once more \cite[Theorem 5.3.3]{EnPr}). 
Now, \cite[Theorem 5.3.3]{EnPr} implies that we have
$$I(u/v) \cong \prod_{q \neq p} \mathbb{Z}_q^{r_q},$$
where $p=\mathrm{char}(Kv)$ and for each prime $q \neq p$, $r_q$ is the $\mathbb{F}_q$-dimension 
of the quotient $vK/q(vK)$. Since $K$ is dp-minimal,the index $[vK:q(vK)]$ and thus $r_q$
is finite for all primes $q$. By \cite[Lemma 5.2.6]{EnPr}, there is an exact sequence
$$1 \ to I(v/w) \to G_K \to G_{Kv} \to 1.$$
Now, since $Kv$ is bounded, $G_{Kv}$ admits only finitely many continuous quotients of index $n$
for each $n \in \mathbb{N}$, and the same holds for $I(u/v)$. Thus,
by Galois correspondence, $K$ is bounded.
\item In case $\mathrm{char}(Kv)=p\neq \mathrm{char}(K)$, let $w$ denote the finest coarsening of $v$
of residue characteristic $0$, i.e.~the valuation corresponding to the convex subgroup of 
$vK$ generated by $v(p)$. Since we also have $[wK:n(wK)]$ finite for all $n \in \mathbb{N}$,  
the arguments given in case (1) imply that it suffices to show that $Kw$
is bounded. Let $\bar{v}$ denote the valuation induced by $v$ on $Kw$. It follows from the 
characterization of $v$ that $\bar{v}$ is defectless,
the value group $\bar{v}(Kw)$ is $p$-divisible and the residue field $(Kw)\bar{v}=Kv$ is
algebraically closed. Thus, $Kw$ admits no Galois extensions of degree divisible by $p$
(and hence the ramification group is trivial)
and is thus $Kw$ is bounded by the same reasoning as in (1).
\end{enumerate}

Now, assume Shelah's
conjecture holds and let $K$ be an infinite strongly dependent field. 
By Theorem \ref{T:main}, $K$ admits again a defectless henselian valuation 
$v$ with residue field
$Kv$ algebraically closed or elementarily equivalent to a local field of
characteristic $0$. Moreover, $vK$ is strongly dependent and in case 
$\mathrm{char}(Kv)=p \neq 0$, we have $[-v(p), v(p)] \subseteq
p(vK)$. In particular, if $\mathrm{char}(K)=p>0$, then $vK$ is $p$-divisible.
By Johnson's Theorem (see \cite[Theorem 9.7.2]{JohnDPMin}
or Theorem \ref{classification} above), 
$K$ is dp-minimal if we have $[vK : n (vK)] < \infty$ for all 
$n \in \mathbb{N}$. 
Thus, it suffices to 
show that if $[vK : n (vK)] = \infty$ for
some $n \in \mathbb{N}$, then $K$ is not bounded. This is a standard argument:
Assume that we have $[vK :n (vK)] = \infty$ for
some $n$, then we get $[vK :q(vK)] = \infty$ for some prime $q$.
Note that we necessarily have $q \neq \mathrm{char}(K)$.
If $K$ does not contain a primitive $q$th root of unity, consider the (Galois)
extension
$L=K(\zeta_q)$ and let $w$ denote the unique prolongation of $v$ to $L$.
Since $[wL:vK]$ is finite, we also have $[wL :q(wL)] = \infty$. 
This implies that the index $[L^\times: (L^\times)^q]$ is infinite.
By Kummer theory (see \cite[Ch. VI \S 8]{Lanalg}), there is a one-to-one 
correspondence between the set of 
finite 
subgroups of $L^\times/(L^\times)^q$ (of which there are infinitely many) 
and the cyclic extensions of $L$ of degree $q$.
Thus, $L$ has infinitely many Galois extensions of degree $q$ and hence $K$
has infinitely many Galois extensions of degree $[L:K]\cdot q$.
\end{proof}

\noindent\textbf{Acknowledgement}  We would like to thank S. Anscombe for allowing us to include the discussion concluding this note, for her comments and suggestions.

\bibliographystyle{plain}
\bibliography{harvard}
\end{document}